\documentclass[11pt]{amsart}
\usepackage{amsmath}
\usepackage{amssymb}
\usepackage{amsfonts}
\usepackage{mathrsfs}
\usepackage{graphicx}
\usepackage{epsfig}

\theoremstyle{plain}

\newtheorem{proposition}{Proposition}

\newtheorem{coro}[proposition]{Corollary}
\newtheorem{lemma}[proposition]{Lemma}

\newtheorem{theoalph}{Theorem}

\theoremstyle{definition}
\newtheorem{definition}{Definition}

\theoremstyle{remark}

\newtheoremstyle{citing}% name
  {3pt}%      Space above, empty = `usual value'
  {3pt}%      Space below
  {\itshape}% Body font
  {}%         Indent amount (empty = no indent, \parindent = para indent)
  {\bfseries}% Thm head font
  {.}%        Punctuation after thm head
  {.5em}%     Space after thm head: " " = normal interword space;
  {\thmnote{#3}}% Thm head spec

\newcommand{\sL}{\mathscr{L}}

\newcommand{\cK}{\mathcal{K}}

\newcommand{\cR}{\mathcal{R}}

\newcommand{\cU}{\mathcal{U}}
\newcommand{\cV}{\mathcal{V}}
\newcommand{\cW}{\mathcal{W}}

\theoremstyle{citing}
\newcommand{\eps}{\varepsilon}

\newcommand{\Shape}{\textrm{Shape}}
\newcommand{\Back}{\textrm{Back}}

\newcommand{\hK}{\widehat{K}}

\newcommand{\hU}{\widehat{U}}
\newcommand{\hV}{\widehat{V}}

\newcommand{\tB}{\widetilde{B}}

\newcommand{\tD}{\widetilde{D}}

\newcommand{\tU}{\widetilde{U}}
\newcommand{\tV}{\widetilde{V}}

\DeclareMathOperator{\comp}{Comp}

\newcommand{\Crit}{\textrm{Crit}}

\newcommand\dist{\mathop{{\rm dist}}}
\renewcommand\int{\mathop{{\rm int}}}
\newcommand\diam{\mathop{{\rm diam}}}
\renewcommand\mod{\mathop{{\rm mod}}}

\newcommand{\C}{\mathbb{C}}
\newcommand{\R}{\mathbb{R}}
\newcommand{\Z}{\mathbb{Z}}
\newcommand{\bad}{\textrm{bad}}
\newcommand{\sep}{\textrm{sep}}
\begin{document}

\bibliographystyle{plain}

\title[Hyperbolicity Assumptions]{On non-uniform hyperbolicity assumptions in one-dimensional dynamics}
%\thanks{2000 {\it Mathematics Subject Classification:} Primary
% 37F35.}
%\footnotetext{{\it Key words and phrases:} Large derivatives,
%backward contraction,  polynomials.}}

\date{November 2, 2009}

\author{ Huaibin Li ;  Weixiao Shen}
%\footnote{ WS is supported by the ``Bai Ren Ji Hua" project of the
%CAS. }}

\maketitle

\begin{abstract}
We give an essentially equivalent formulation of the backward
contracting property, defined by Juan Rivera-Letelier, in terms of
expansion along the orbits of critical values, for complex
polynomials of degree at least $2$ which are at most finitely
renormalizable and have only hyperbolic periodic points, as well as
all $C^3$ interval maps with non-flat critical points.
\end{abstract}

\section{Introduction}\label{sec:intro}
In the context of one-dimensional dynamics, widely adopted
non-uniform hyperbolicity conditions involve expansion along the
orbits of critical values, such as the Collet-Eckmann condition, the
summability conditions and the large derivatives condition,
see~\cite{CE, NS, GS, BSS, BRSS} among others. Recently,
Rivera-Letelier~\cite{R} introduced a new notion called {\em
backward contraction} which serves as a different type of
non-uniform hyperbolicity condition. This condition is more
convenient to use as it follows immediately that the first return
maps to suitably chosen small neighborhoods of critical points have
good combinatorial and geometric properties. For instance, this
notion plays an important role in the work~\cite{BRSS}.
%It turns out that many dynamical properties
%hold under this new assumption, see~\cite{R, BRSS, LS, RS}.

%It has been noticed in~\cite{R}[Theorem A]

%It has been realized that backward contraction follows from
%sufficient strong expansion along the orbits of critical values,
%see~\cite[Theorem A]{R} and ~\cite{BRSS}.
It has been realized that sufficient expansion along the orbits of
critical points often implies backward contraction,
see~\cite[Theorem A]{R} and ~\cite[Theorem 1]{BRSS}. In this paper,
we study further the relation between these two types of
non-uniformly hyperbolicity conditions. For notational definiteness,
we shall mainly work on complex maps and leave the argument for
interval maps in Appendix ~\ref{sec:interval}.

Given a complex polynomial $f$, let $\textrm{Crit}(f)$ denote the set of critical points of $f$ in
$\C$, let $J(f)$ denote the Julia set of $f$, and let
$$\Crit'(f)=\Crit(f)\cap J(f).$$  For every $z\in \Crit(f)$ and
$\delta >0$ we denote by $B(f(z), \delta)$ the Euclidean ball of
radius $\delta$ and centered at $f(z)$. Moreover, we denote by
$\widetilde{B}(z,\delta)$ the connected component of
$f^{-1}(B(f(z),\delta))$ that contains $z.$
\begin{definition}\label{def:bc}
{\rm Given a constant $r>1,$ we say that $f$ satisfies the {\em
backward contraction property with constant $r$} ($f\in BC(r)$ in
short) if there exists $\delta_{0}>0$ such that for every $c\in
\Crit'(f),$ every $0<\delta \leq \delta_{0},$ every integer $n\geq1$
and every component $W$ of $f^{-n}(\widetilde{B}(c,r\delta)),$  we
have that
$$ \dist (W,CV(f))\leq \delta \quad \Rightarrow \quad \diam(W)<\delta $$
where $CV(f)=f(\Crit(f))$. If $f\in BC(r)$ for all $r>1, $ we will
say that $f\in BC(\infty)$.}
\end{definition}

In~\cite[Theorem A]{R}, the author showed that for a complex
polynomial (or more generally a rational map), if
$$\sum_{n=0}^\infty \frac{1}{|Df^n(f(c))|}<\infty$$
holds for all $c\in \Crit'(f)$, then $f$ satisfies $BC(\infty)$.

\begin{definition}\label{def:ld}
We say that a polynomial $f$ satisfies the {\em large derivative
condition with constant $K$} ($f\in LD(K)$ in short) if there exists
a neighborhood $V$ of $\Crit'(f)$ such that for each $c\in
\Crit'(f)$ and $n\ge 1$ with $f^n(c)\in V$, we have
$$|Df^n(f(c))|\ge K.$$ If $f\in LD(K)$ for all $K>0, $ we will
say that $f\in LD(\infty)$.
\end{definition}
Obviously, given a polynomial $f$, if for every $c\in \Crit'(f),$ we
have $$\lim_{n\rightarrow \infty}|Df^n(f(c))|= \infty$$ then $f$
satisfies $LD(\infty).$

This definition was given first in~\cite{BRSS} for $C^3$ interval
maps with non-flat critical points and with all periodic points
hyperbolic repelling, where it was proved that for such maps, $f\in
LD(\infty)$ implies $f\in BC(\infty)$, where the properties $LD(K)$
and $BC(r)$ are defined as above except that we use the standard
metric on the interval and use $\Crit'(f)=\Crit(f)$. The proof uses
a special tool in real one-dimensional dynamics, namely the {\em
one-sided Koebe principle}, which has no complex analogy.

We shall prove the following in \S\ref{sec:ld2bc}:
\begin{theoalph}[Large derivative implies backward contraction]\label{thm:ld2bc} For each
integer $d\ge 2$, there exists $K_0=K_0(d)>0$ such that if $f$ is a
polynomial of degree $d$ which is at most finitely renormalizable
and has only hyperbolic periodic points and if $f$ satisfies
$LD(K_0r)$ for some $r>1$, then $f$ satisfies $BC(r)$.
\end{theoalph}
Recall that a {\em renormalization} of $f$ is a map $f^s: U\to V$,
where $s$ is a positive integer and $V\Supset U$ are Jordan disks,
such that the following hold:
\begin{itemize}
\item $f^s: U\to V$ is proper;
\item $U$ contains a critical point in $J(f)$;
\item the following set (called the filled Julia set of $f^s: U\to V$) is connected:
$$\{z\in U: f^{sn}(z)\in U\mbox{ for all } n=1,2,\ldots\};$$
\item for each $c\in\Crit(f)$, there exists at most one $j\in \{0,1,\ldots, s-1\}$ with $c\in f^j(U)$;
\item $U\not\supset J(f)$.
\end{itemize}
We say that $f$ is infinitely renormalizable if there exists a sequence of renormalizations $f^{s_k}: U_k\to V_k$ such that $s_k\to\infty$.

%\marginpar{think}

As a consequence of the Schwarz lemma, we shall also prove the
following converse statement in \S\ref{sec:bc2ld}:
\begin{theoalph}[Backward contraction implies large
derivative]\label{thm:bc2ld} Let $f$ be a polynomial of degree at
least $2$. There exists a constant $r_0$ depending on the maximal
critical order of $f$ such that if $f$ satisfies $BC(Kr_0)$ for some
$K>1$, then $f$ satisfies $LD(K)$.
\end{theoalph}

Combining these two theorems, we obtain that
\begin{coro} Let $f$ be a polynomial which is at most
finitely renormalizable and has only hyperbolic periodic points.
Then $f$ satisfies $BC(\infty)$ if and only if $f$ satisfies
$LD(\infty)$.
\end{coro}

Our proof of the Theorem~\ref{thm:ld2bc} is based on the following
complex bounds established in~\cite{KS}, which depends heavily on
the recent analytic result~\cite{KL} and the enhanced nest
construction~\cite{KSS}. See also~\cite{QY} for the case of Cantor
Julia sets. Since the precise form is not stated explicitly
in~\cite{KS}, we include a proof of the proposition in the Appendix
~\ref{sec:bounds} for completeness.

\begin{proposition}[Complex bounds] \label{prop:aprioribounds}
Assume that $f$ is a complex polynomial of degree $d\ge 2$ which is
at most finitely renormalizable and has only hyperbolic periodic
points. Then there exists $\rho_0=\rho_0(d)>0 $ such that for each $c\in
\Crit'(f),$ there exists an arbitrarily small $\rho_0$-nice
topological disk which contains $c$.
\end{proposition}
Recall that an open set $  V \subset \mathbb{C}$ is called {\em
nice} if $f^{n}(\partial V) \cap V = \emptyset $ for all $n \geq 0.$
We say that $V$ is {\em $\rho$-nice} if for each return domain $U$
of $V$ , there is an annulus $A\subset V\setminus \overline{U}$ such
that $U$ is contained in the bounded component of $\C\setminus A$
and such that $\mod(A)\ge \rho.$

In appendix \ref{sec:interval}, we shall prove similar results for
interval maps.

\noindent {\bf  Terminology and notation:}\\
\indent A {\em topological disk} means a simply connected domain in
$\mathbb{C}.$ An annulus $A$ is a doubly connected domain in $\C$,
and the modulus $\mod(A)$ of $A$ is defined to be $\log R/r$, where
$A$ is conformal isomrophic to the round annulus $\{r<|z|<R\}$.
Given an open set $V$ and a set $E$ with $\overline{E}\subset V$,
let
$$\mod(V; E)=\sup_A \mod (A),$$
where the supremum is taken over all annuli $A$ with the property
that $A\subset V\setminus \overline{E}$ and $E$ is contained in the
bounded component of $\C\setminus A$.

Given a nice set $V ,$ let $$D(V)= \{z\in \mathbb{C}: f^{n}(z) \in V
\mbox{ for some }n\geq 1 \}.$$

The {\em first entry map } $R_{V}: D(V)\rightarrow V$ is defined as
$z \mapsto f^{k(z)}(z),$ where $k(z),$ called the entry time of $z$
into $V,$ is the minimal positive integer such that $f^{k(z)}(z) \in
V.$ Since $f$ is continuous and $V$ is nice, $k(z)$ is constant in
any component of $D(V).$ A component of $D(V)$ is called an {\em
entry domain.} The map $R_{V}\mid _{D(V)\cap V} $ is called the {\em
first return map} of $V,$  and a component of $D(V)\cap V$ is called
a {\em return domain.} For any $x\in D(V)$, let $\sL_x(V)$ denote
the entry domain which contains $x.$ Moreover, for $x\in D(V)\cup
V,$ let $\widehat{\sL}_x(V)=\sL_x(V)$ if $x\in D(V)\setminus V$, and
let $\widehat{\sL}_x(V)= V$ if $x\in V.$

{\bf Acknowledgement.} We would like to thank Juan Rivera-Letelier
for reading carefully a first version of the manuscript and several
valuable discussions.

\section{Large derivative implies backward contraction}
\label{sec:ld2bc} The goal of this section is to prove
Theorem~\ref{thm:ld2bc}. So let us fix an integer $d\ge 2$ and let
$f$ denote a polynomial of degree $d$. We assume throughout this
section that $f$ is at most finitely renormalizable and has only
hyperbolic periodic points. For each critical point $c$, let
$\ell_c$ denote the order of $c$, and let
$\ell_{\max}(f)=\max\{\ell_c:c\in\Crit'(f)\}.$

In the following we shall not state explicitly dependence of
constants on the degree $d$. So a constant depending only on $d$
will be called {\em universal}.

\subsection{Preparation}
We shall use the following variation of the Koebe principle.

\begin{lemma}\label{lem:koebevariation}
For any $\rho>0$ and $N\ge 1$, there exists $A_0=A_0(\rho, N)>0$
such that the following holds. Let $V\supset D, U\supset E$ be
bounded topological disks and let $s$ be a positive integer with the
following properties:
\begin{itemize}
 \item $\mod(V; D)\ge \rho$;
 \item $U$ is a component of $f^{-s}(V)$;
 \item the degree of $f^s: U\to V$ is at most $N$.
 \item $E$ is a connected component of $f^{-s}(D)$.
\end{itemize}
Then for any $x\in E$,
$$|Df^s(f(x))|\le A_0\frac{\diam (f(D))}{\diam (f(E))},$$
provided that $\diam (D)$ is small enough.
\end{lemma}
\begin{proof}
Certainly we only need to prove the lemma in the case $\rho<1$. By
considering a suitable restriction of the map $f^s: U\to V$ we may
assume that $\mod(V; D)=\rho$.

Let $\varphi: f(U)\to \mathbb{D}$, $\psi: V\to \mathbb{D}$ be
Riemann mappings with $\varphi(f(x))=0$ and $\psi(f^s(x))=0$, where
$\mathbb{D}$ denotes the unit disk in $\C$. Then $F:=\psi\circ
f^{s-1}\circ \varphi^{-1}$ is a holomorphic map from $\mathbb{D}$
into itself and $F(0)=0$. Thus by the Schwarz lemma, $|F'(0)|\le 1$.
Since $F$ is proper and $\deg (F)\le N$, we have
$$\mod(f(U); f(E))\ge \frac{1}{\deg (F)}\mod(V; D)\ge \frac{\rho}{N}.$$
By the Koebe distortion theorem, it follows that
$$|\varphi'(f(x))|\le C_1 \diam (f(E))^{-1},$$ where $C_1=C_1(\rho, N)$
is a constant. Since $$\mod(\mathbb{D}; \psi (D))=\mod(V;
D)=\rho<1,$$ $\diam (\psi(D))\ge e^{-1}$.
Again by the Koebe distortion theorem, we obtain that
$$|(\psi^{-1})'(0)|\le C_2\diam (D),$$
where $C_2=C_2(\rho)$.
Finally, provided that $\diam (D)$ is small enough, we have
$$|Df(f^s(x))|\le C_3\diam (f(D))/\diam (D),$$
where $C_3>0$ is a constant depending only on the degree of $f$.
Combining all these estimates, we obtain
\begin{align*}
|Df^s(f(x))|& =|Df(f^s(x))||F'(0)| |\varphi'(f(x))|
|(\psi^{-1})'(0)| \\
& \le C_1C_2C_3\frac{\diam (f(D))}{\diam (f(E))}.
\end{align*}
Thus the lemma holds with $A_0=C_1C_2C_3$.
\end{proof}

Given an bounded open set $\Omega\subset\C$ and $z\in \Omega$, let
$$IR(\Omega, z)=\inf_{w\in\partial\Omega} d(z, w), \,\,\, OR(\Omega, z)
=\sup_{w\in \partial\Omega} d(z, w)$$ and
$$\Shape (\Omega, z)= \frac{OR(\Omega, z)}{IR(\Omega, z)}.$$

We shall use the following procedure to construct nice sets with
bounded shape. Given a nice topological disk $V$, a point $z_0\in V$
and a constant $\lambda>0$, let $B_V(z_0,\lambda)$ be the hyperbolic
ball centered at $z_0$ and of radius $\lambda$ (in the hyperbolic
Riemann surface $V$), let $V_*[z_0,\lambda]$ be the union of
$B_V(z_0,\lambda)$ and all the return domains of $V$ that intersect
this set, and let $V[z_0,\lambda]$ be the filling of
$V_*[z_0,\lambda]$, i.e. the union of $V_*[z_0,\lambda]$ and the
bounded components of $\C\setminus V_*[z_0,\lambda]$. Clearly, for
each $n\ge 1$,
$$f^n(\partial V[z_0,\lambda])\cap V=\emptyset.$$

\begin{lemma}\label{lem:filling}
For each $\rho>0$ and $\lambda>0$ there exist $M>1 $ and $\rho'>0$
such that if $V$ is a $\rho$-nice topological disk, then for any
$z_0\in V$, $V[z_0,\lambda]$ is $\rho'$-nice and $\Shape
(V[z_0,\lambda], z_0)\le M$.
\end{lemma}
\begin{proof} Write $U=V[z_0,\lambda]$.
Since $V$ is $\rho$-nice, the hyperbolic diameters of return
domains of $V$ in $V$ are uniformly bounded from above by a constant
depending only on $\rho$. It follows that the hyperbolic diameter of
$V_*[z_0,\lambda]$, hence that of $V[z_0,\lambda]$ in $V$ is bounded
from above by a constant depending only on $\rho$. Consequently,
$\mod(V; {U})$ is bounded away from zero. By the Riemann mapping
theorem and the Koebe distortion theorem, we obtain that $\Shape
(U,z_0)$ is bounded from above.

Let us prove that $U$ is $\rho'$-nice. Indeed, each return domain
$W$ of $U$ is contained in a return domain $W'$ of $V$. The first
return map $R_V$ to $V$ maps $W$ into either $U$ or a return domain
of $V$. In both cases, we have that $\mod(V; R_V(W))$ is bounded
away from zero. Since $R_V|W'$ has bounded degree and since
$W'\subset U$, we obtain that $\mod(U; {W})$ is bounded away from
zero.
\end{proof}

We say that $(\cV', \cV)$ is an {\em admissible pair} of neighborhoods of a set $A\subset \Crit'(f)$
if the following hold:
\begin{itemize}
\item $\cV \subset \cV'$;
\item $\cV$ (resp. $\cV'$) is nice and each component of $\cV$ (resp. $\cV'$)
is a topological disk containing exactly one
point of $A$;
\item for all $n\ge 1$ and $a\in A$,
$$f^n(\partial \cV_a)\cap \cV_a'=\emptyset,$$
where $\cV_a$ (resp. $\cV_a'$) denotes the component of $\cV$ (resp.
$\cV'$) which contains $a$.
\end{itemize}
We say that $(\cV', \cV)$ is
called {\em $\rho$-bounded} if for every $a\in A,$ we have
$$\mod(\cV_a'; \cV_a)\ge \rho.$$

For each $c_0\in\Crit'(f)$, let
$$\Back(c_0)=\{c\in\Crit'(f): \overline{\{f^n(c):n\ge 0\}}\ni
c_0\}\ni c_0.$$ We shall need the following lemma which is essentially \cite[Lemma
6.3]{KSS}.
\begin{lemma} \label{lem:admdeg}
Let $c_0\in\Crit'(f)$ and let $(\cV', \cV)$ be an
admissible pair of neighborhoods of $\Back(c_0)$. Assume
that $\max_{c\in \Back(c_0)} \cV'_c$ is sufficiently small. Let $U$
be an entry domain of $\cV$ with entry time $s$, let
$c_1\in\Back(c_0)$ be such that $f^s(U)=\cV_{c_1}$, and let $U'$ be
the component of $f^{-s}(\cV'_{c_1})$ that contains $U$, then the
degree of $f^s: U'\to \cV'_{c_1}$ does not exceed a universal
constant $N_0$.
\end{lemma}
\begin{proof} Assume that $\diam(\cV'_{c_1})$ is small enough, so that for any $c\in\Crit'(f)$, if there exists $k\ge 1$ such that $f^k(c)\in \cV_{c_1}$ then $c\in\Back(c_1)\subset \Back(c_0)$.
For $c\in\Back(c_0)$, and any $i\in\{1,2,\ldots, s-1\}$, $f^i(U')\not\ni c$, for otherwise,
we would have $f^{i}(U')\subset \cV_{c}$, contradicting the hypothesis that $s$ is the entry time of $U$ into $\cV$.
\end{proof}
\subsection{Nested nice sets}
In this section, we shall prove the following proposition, which is
a crucial step of the proof of Theorem~\ref{thm:ld2bc}.

\begin{proposition}[Nested nice sets]\label{prop:nicenets}
There exist universal constants $K_*>0$, $\kappa_*>0$ and $C_*>0$
such that if $f$ satisfies $LD(K)$ with $K\ge K_*$, then for each
$c_0\in \Crit'(f)$, there exists an infinite sequence of nice
topological disks
$$V_1\supset V_2\supset \cdots$$ that contain $c_0$ such that
the following hold:
 \begin{enumerate}
 \item[1.] $\diam (V_k)\to 0$ as $k \to
 \infty$;
 \item[2.] for each $k\ge 1$, $\Shape (f(V_k), f(c_0))\le C_*$;
 \item[3.] for each $k\ge 1$, $\kappa_*\diam (f(V_k))\le \diam(f(V_{k+1}))$;
 \item[4.] for each $k\ge 1$, and any $c\in\Back(c_0)$, we have
 $$\diam (f(\sL_{c}(V_k)))\le \frac{C_*}{K} \diam (f(V_k)).$$
 \end{enumerate}
\end{proposition}

To prove this proposition, we need a few lemmas.

Given a point $c\in \Crit'(f),$ a topological disk $W\ni c$ is
called {\em $(\rho, M)$-bounded} if $W$ is $\rho$-nice and $W$
satisfies $\Shape (W, c)\le M.$ Sometimes we say that $W$ is
uniformly bounded if it is $(\rho, M)$-bounded for some universal
constants $\rho$ and $M$.

\begin{lemma}\label{lem:deepV1}
Given any $\rho>0$ and $M>0$,  there exist $K_1=K_1(\rho)>0$ and
$C=C(\rho, M)>1$ such that if $f$ satisfies $LD(K)$ with $K\ge K_1$,
the following holds. Given $c_0\in \Crit'(f)$ and a
$(\rho,M)$-bounded puzzle piece $\hV\ni c_0$ with $\diam (\hV)$
sufficiently small, then $V:=\hV[c_0,1]$ satisfies the following
properties:
\begin{enumerate}
\item[(i)] for each $c\in\Back(c_0)$, we have
$$\diam (f(\sL_c(V)))\le \frac{C}{K} \diam (f(V)).$$
\item[(ii)] For each return domain $U$ of $V$, either
$$U \subset \tB(c_0, CK^{-1} \diam (f(V)))$$ or $$\diam (U)\le C
\dist (U, c_0).$$
\end{enumerate}
\end{lemma}
\begin{proof}
Let us first construct an admissible pair $(\widehat{\mathcal{V}},
\mathcal{V})$ of neighborhoods of $\Back(c_0)$ as follows. Let
$\mathcal{D}$ be the collection of return domains of $\hV$ outside
$V$ and $V$ itself. Let $c_1, c_2,\ldots, c_m$ be the critical
points in $\Back(c_0)\setminus \{c_0\}$. For each $i=0,1,\ldots, m$,
let $\hV^{i}=\widehat{\sL}_{c_i}(\hV)$ and let $t_i$ be the landing
time of $c_i$ into $\hV$.  Let $D_i$ be the element of $\mathcal{D}$
that contains $f^{t_i}(c_i)$ and let
$V^{i}=\comp_{c_i}(f^{-t_i}(D_i))$.(So $t_0=0$, $\hV^0=\hV$, and
$V^0=V$.) Moreover, let
$$\widehat{\mathcal{V}}=\bigcup_{i=0}^m \hV^i,\hspace{2mm}
%\widetilde{\mathcal{V}}=\bigcup_{i=0}^m \tV^i,\hspace{2mm}
\mathcal{V}=\bigcup_{i=0}^m V^i.$$

For each $i=1,2,\ldots, m$, the degree of the first entry map
$f^{t_i}: \hV^i \to\hV$ is bounded from above by  $N:=\deg (f)^m$.
Since $\hV$ is $\rho$-nice, by Lemma~\ref{lem:filling}, there exists
$\rho'=\rho'(\rho)>0$ such that $\mod(\hV; {D})> \rho'$,
$D\in\mathcal{D}$.
%is bounded away from zero,
Thus
$$\mod(\hV^i; {V^i})\ge \rho'/N.$$
Provided that $\diam(\hV)$ is small enough, by the assumption that
$f$ satisfies $LD(K)$, we have $$|Df^{t_i}(f(c_i))|\ge K, i=1,2,\ldots, m.$$ Thus, by
Lemma~\ref{lem:koebevariation}, we have
$$\diam (f(V^i))\le \frac{A_0(\rho', N)}{|Df^{t_i}(f(c_i))|} \diam (f(D_i))\le \frac{A_0(\rho', N)}{K} \diam (f(\hV)).$$
Since $\Shape (\hV, c_0)\le M$, there exists $M'$ depending only on
$M$ such that $$\diam (f(\hV))\le M'\diam (f(V)).$$ Since
$\sL_{c_i}(V)\subset V^i$, we obtain that
\begin{equation}\label{eqn:sLc_iV}
\diam (f(\sL_{c_i}(V)))\le \diam (f(V^i))\le C_1K^{-1} \diam (f(V)),
\end{equation}
holds for all $i=1,2,\ldots, m$, where $C_1=A_0(\rho', N)M'$ is a
constant.

Assume now that $f$ satisfies $LD(K)$ with $K\ge K_1:=A_0(\rho'/2,
N)$.

 \noindent{\bf
Claim.} There exists a constant $C=C(\rho, M)$ such that if
$\diam(\hV)$ is small enough, then the following holds: for each
return domain $U$ of $\mathcal{V}$ with $U\subset V$,
$$\text{either }U\subset \tB(c_0, CK^{-1}\diam (f(V))) \text{ or }\diam
(U)\le C\dist (U, c_0).$$

To prove this claim, let $\tD_i$ be the topological disk with
$D_i\subset \tD_i\subset \hV$ and with $\mod(\hV;\tD_i)=\mod (\tD_i;
D_i)=\mod(\hV; D_i)/2$, and let $\tV^i$ be the component of
$f^{-t_i}(\tD_i)$ that contains $c_i$. Then $$\mod(\hV^i; \tV^i)\ge
\mu:=\frac{\rho'}{2N}, \,\, \mod(\tV^i; V^i)\ge \mu$$ and
$$\diam (f(\tV^i))\le \frac{A_0(\rho'/2, N)}{K} \diam (f(\tD_i))\le \diam (f(\hV)).$$
For each return domain $U$ of $\mathcal{V}$ with return time $s$ and
with $f^s(U)=V^i$, let $\hU$ and $\tU$ be the components of
$f^{-s}(\hV^i)$ and $f^{-s}(\tV^i)$ that contain $U$ respectively.
By Lemma~\ref{lem:admdeg},  the degree of $f^s: \hU\to \hV^i$ is
bounded from above by $N_0$, $\mod(\tU; {U})\ge \rho'/(2NN_0)$. If
$\tU\not\ni c_0$, then $\diam (U)/\dist(U, c_0)$ is bounded from
above by a constant depending on $\rho$. If $\tU\ni c_0$, then by
Lemma~\ref{lem:koebevariation}, we obtain $$\diam(f(\tU))\le
A_0(\mu, N) K^{-1} \diam (f(\tV^i))\le A_0(\mu, N) K^{-1}\diam
(f(\hV)),$$ which implies that $U\subset \tB(c_0,
CK^{-1}\diam(f(V)))$, where $C=A_0(\mu, N) M'$ depends only on
$\rho$ and $M$. This proves the claim.

Since each return domain of ${V}$ is contained in a return domain of
$\mathcal{V}$, the statement (ii) follows from the claim. Moreover,
the claim implies that $\diam (f\sL_{c_0}(V))\le CK^{-1}\diam
(f(V))$, which together with (\ref{eqn:sLc_iV}) implies (i).
\end{proof}

\begin{lemma}\label{lem:3deep}
For each $\rho>0$, there exists $K_2=K_2(\rho)>0$ such that if $f$
satisfies $LD(K_2)$, then the following holds. If $W$ is a
$\rho$-nice puzzle piece that contains a critical point
$c\in\Crit'(f)$ and if $\diam (W)$ is small enough, then either
$\sL_c^3 (W)=\emptyset$, or
$$\mod (W; {\sL_c^3(W)})\ge 1.$$
\end{lemma}
\begin{proof} Let $\hV=W[c, 1]$, $V=\hV[c,1]$. By Lemmas~\ref{lem:filling},
there exists $\rho'>0$ and $M>1$ depending only on $\rho$ such that
both $\hV$ and $V$ are $(\rho', M)$-bounded. By
Lemma~\ref{lem:deepV1}, if $f$ satisfies $LD(K)$ with $K\ge
K_1(\rho')$, then
$$\diam (f(\sL_{c}(V)))\le CK^{-1} \diam (f(V)),$$
where $C=C(\rho', M)$ (depending only on $\rho$).  Thus $\mod(V;
\sL_{c}(V))\ge 1$ provided that $K$ is large enough. Since
$\sL_c^3(W)\subset \sL_c(V)$, the lemma follows.
\end{proof}

%Recall that $(V', V)$ of two nice sets is called {\em nice pair} if
%$\overline{V}\subset V'$ and $f^{n}(\partial V) \cap V' = \emptyset
%$ for all $n \geq 1.$
\begin{lemma}\label{lem:out2innbounds}
For each $\rho>0$, there exists $K_3=K_3(\rho)>0$ such that if $f$
satisfies $LD(K_3)$, then the following holds. Let $(V', V)$ be an
admissible pair of neighborhoods of some $c_0\in\Crit'(f)$ such that
$V'$ is $\rho$-nice and $\mod(V'; V)\ge 1$. Assume furthermore that
$\diam (V')$ is small enough. Then $V$ is $\rho_*$-nice, where
$\rho_*>0$ is a universal constant (independent of $\rho$).
\end{lemma}
\begin{proof} Let $U$ be a return domain of $V$ with return time
$s$. We need to find a universal bound for $\mod(V;U)$. Let $U_j'$
be the component of $f^{-(s-j)}(V')$ which contains $f^j(U)$,
$j=0,1,\ldots,s$. Then $U_0'\subset V$ since $(V', V)$ is an
admissible pair. If for each $c\in \Crit'(f)$,
\begin{equation}\label{eqn:criticalhitting}
\#\{0\le j<s: U_j'\ni c\}<5,
\end{equation}
 then $f^s: U_0'\to V'$ has uniformly
bounded degree, hence $\mod(V; {U})$ is bounded from below by a
positive constant. So assume that (\ref{eqn:criticalhitting}) fails
for some $c\in\Crit'(f)$. Then clearly $c\in
\Back(c_0)\setminus\{c_0\}$ and there exists $1\le s_1<s$ such that
$U_{s_1}'\subset \sL_c^4(V')$. Therefore, there exists a minimal
integer $s'\in \{1,2,\ldots, s-1\}$ such that $U_{s'}\subset
\sL_{c'}^4 (V')$ for some $c'\in \Back(c_0)\setminus\{c_0\}$. Let
$W_j'$ be the component of $f^{-(s'-j)} (\sL_{c'}(V'))$ which
contains $f^j(U)$. By the minimality of $s'$, we have that for each
$c\in\Crit'(f)$,
$$\#\{0\le j<s': W_j'\ni c\}<4.$$
Thus $f^{s'}: W_0'\to \sL_{c'}(V')$ has uniformly bounded degree.
Since $V'$ is $\rho$-nice, $\sL_{c'}(V')$ is $\rho'$-nice, where
$\rho'>0$ is a constant depending only on $\rho$. By
Lemma~\ref{lem:3deep}, if $f$ satisfies $LD(K_3)$ with
$K_3=K_3(\rho)=K_2(\rho')$, then
$$\mod(\sL_{c'}(V'); {f^{s'}(U)})\ge
\mod(\sL_{c'}(V'); {\sL_{c'}^4(V')})\ge 1.$$ It follows that
$\mod(V; {U})\ge \mod(W_0'; U)$ is bounded from below by a positive
constant.
\end{proof}

\begin{lemma}\label{lem:recursive}
Given any $\rho>0$ and $M>0$ there exist constants $\hK=\hK(\rho,
M)>0$ and $\kappa\in (0,1)$ such that if $f$ satisfies $LD(\hK)$,
then the following holds. Given a $(\rho, M)$-bounded puzzle piece
$\hV$ which contains $c_0\in\Crit'(f)$ and such that $\diam(\hV)$ is
small enough, there exists a $(\rho_1, M_1)$-bounded puzzle piece
$\hV_1\ni c_0$ such that
$$\kappa \diam (f(\hV))\le \diam (f(\hV_1))\le \frac{1}{e} \diam (f(\hV)),$$
where $\rho_1>0, M_1>1$ are universal constants (independent of
$\rho, M$).
\end{lemma}

\begin{proof} Assume that $f$ satisfies $LD(K)$ with a large constant $K$. By Lemma~\ref{lem:deepV1},
each return domain $U$ of $V:=\hV[c_0, 1]$ satisfies one of the
following: $U\subset \tB(c_0, CK^{-1}\diam (f(V)))$ or $\diam (U)\le
C\dist(U, c_0)$, where $C=C(\rho, M)$ is a constant. By
Lemma~\ref{lem:filling}, $\Shape (V, c_0)$, hence $\Shape (f(V),
f(c_0))$ is bounded from above by a constant $M'$ depending on
$\rho$ and $\ell_{c_0}$. Let $\eps= (2M'(C+2)e^{2\pi
\ell_{c_0}})^{-1}$ and assume that $K> C/\eps$. Let $\tV$ be the
filling of the union of $\tB(c_0, \eps \diam(f(V)))$ and all the
return domain of $V$ that intersect $\tB(c_0, \eps \diam(f(V)))$.
Then, $\tV\subset \tB(c_0, (2M'e^{2\pi
\ell_{c_0}})^{-1}\diam(f(V)))$, hence $\mod(V; \tV)\ge 1$. By
Lemma~\ref{lem:out2innbounds}, we obtain that $\tV$ is
$\rho_*$-nice. Take $\hV_1=\tV[c_0, 1]$. By Lemma~\ref{lem:filling},
$\hV_1$ is $(\rho_1, M_1)$-bounded for some universal constants
$\rho_1$ and $M_1$. The estimate on the diameter of $\diam (f\hV_1)$
follows from the construction.
\end{proof}

\begin{lemma}\label{lem:nicepuzzle}
There exist universal constants $\hK_*>0$ and $\hat{\kappa}_*>0$ such
that if $f$ satisfies $LD(\hK_*)$, then for each $c_0\in \Crit'(f)$,
there exists an infinite sequence of $(\rho_1, M_1)$-bounded nice
topological disks
$$\hV_1\supset \hV_2\supset \cdots$$ that contain $c_0$ such that for each
$k\ge 1$,
$$\hat{\kappa}_*\diam (f(\hV_k))\le \diam(f(\hV_{k+1}))\le e^{-1}\diam (f(\hV_k))$$
\end{lemma}

\begin{proof} It suffices to prove existence of an arbitrarily small
$(\rho_1, M_1)$-bounded nice topological disk $\hV_0\ni c_0$ under
the assumption that $f$ satisfies $LD(K)$ with a large $K$, since
then we may apply Lemma~\ref{lem:recursive} successively to obtain
the desired sequence.

 By Proposition~\ref{prop:aprioribounds}, there
exists an arbitrarily small $\rho_0$-nice topological disk $V\ni
c_0$. Let $\hV=V[c_0, 1]$. By Lemma~\ref{lem:filling}, $\hV$ is
uniformly bounded, so by Lemma~\ref{lem:recursive}, we obtain
existence of $\hV_0$.
\end{proof}

\begin{proof}[Proof of Proposition~\ref{prop:nicenets}] Let $K_*=\max(\hK_*, K_1(\rho_1))$.
Assume that $f$ satisfies $LD(K)$ with $K\ge K_*$. Let $\hV_k$ be as
in Lemma~\ref{lem:nicepuzzle} and let $V_k=\hV_k [c_0, 1]$. Then the
first, second  and third statements hold with suitable choices of
$C_*$ and $\kappa_*$. By Lemma~\ref{lem:filling}, $\Shape
(V_k,c_0)$, hence $\Shape (f(V_k), f(c_0))$, are uniformly bounded.
By Lemma~\ref{lem:deepV1}, the last statement holds.
\end{proof}

\subsection{Proof of Theorem~\ref{thm:ld2bc}}

In~\cite[Section 6]{R}, the author introduced another notion called
{\em univalent pull back condition}, which is closely related to
backward contraction. (A similar notion, $BC^*(r)$, was used in
\cite{BRSS}.)

Given $\delta'>\delta>0$, we say that $f$ satisfies the
$(\delta,\delta')$-univalent pull back condition if for every
$z\in\C$ and every integer $n\ge 1$ such that
\begin{itemize}
\item for each $j=1,2,\ldots, n-1$, $f^j(z)\not\in \bigcup_{c\in\Crit'(f)}\tB(c, \delta)$
\item for some $c\in\Crit'(f)$, we have $f^n(z)\in \tB(c, \delta')$,
\end{itemize}
then $f^n$ maps a neighborhood of $z$ conformally onto $\tB(c,
\delta')$. Given $r>1$, we say that $f$ satisfies the univalent pull
back condition with constant $r$ if for all $\delta>0$ sufficiently
small, $f$ satisfies the $(\delta, r \delta)$-univalent pull back
condition.

The following is \cite[Proposition 6.1 part 2]{R}.
\begin{lemma}\label{lem:up2bc}
There exists a constant $r_0>1$ such that if $f$ satisfies the
univalent pull back condition with constant $rr_0$, where $r>1$ is a
constant, then $f$ satisfies the backward contracting condition with
constant $r$.
\end{lemma}
\begin{proof}[Proof of Theorem~\ref{thm:ld2bc}]
Let $K_0=\max (K_*, 2C_*^2 \kappa_*^{-1} r_0)$, where $K_*>0,$
$C_*>0$ and $\kappa_*>0$ are as in Proposition~\ref{prop:nicenets}.

Fix $r>1$ and assume that $f$ satisfies $LD(K)$ with $K=rK_0$. Let
us prove that $f$ satisfies the univalent pull back condition with
constant $rr_0$ which implies that $f$ satisfies the backward
contraction condition with constant $r$ by Lemma~\ref{lem:up2bc}. To
this end, it suffices to prove that for each $c_0\in\Crit'(f)$ and
each $\delta>0$ small enough, if $U$ is a pull back of
$\tB(c_0,\delta)$ that intersects $\Crit'(f)$, then $\diam (f(U))\le
(rr_0)^{-1}\delta$. Let $V_k$ be as given in
proposition~\ref{prop:nicenets}. For each $\delta\in (0,
(2C_*)^{-1}\diam (V_1))$, there exists a maximal integer $k\ge 1$
such that $\tB(c_0, \delta)\subset V_k$. By part 2 and 3 of the
proposition~\ref{prop:nicenets}, we have
$$\diam (f(V_k))\le 2C_* \kappa_*^{-1}\delta.$$ If $U$ is a pull back of
$\tB(c_0,\delta)$ that contains a critical point $c$, then $U\subset
\sL_c(V_k)$. Thus by part 4 of the proposition~\ref{prop:nicenets},
we obtain
\begin{align*}
\diam (f(U)) &\le \diam (f(\sL_c(V_k))) && \le C_* K ^{-1} \diam
(f(V_k))\\ & \le 2C_*^2 \kappa_*^{-1} K^{-1}\delta && \le
(rr_0)^{-1}\delta.
\end{align*}
The proof is completed.
\end{proof}

\section{Backward contraction implies large derivatives}\label{sec:bc2ld}
In this section, we shall prove Theorem~\ref{thm:bc2ld}. Let $f$ be
a polynomial which satisfies the backward contraction condition with
constant $r>4$. So there exists $\delta_0>0$ such that for each
$\delta\in (0,\delta_0]$ and each $c, c'\in \Crit'(f)$, if $U\ni c$
is a component of $f^{-n}(\tB(c',r\delta))$ then $\diam
(f(U))<\delta.$ We continue to use $\ell_{\max}$ to denote the
maximal order of critical points in the Julia set.

 For each $n\ge 0$ and $c\in\Crit'(f)$, let
$\delta_n=2^{-n}\delta_0$, $V_n^c=\tB(c, \delta_n)$ and let
$V_n=\bigcup_{c\in\Crit'(f)} V_n^c$.

By reducing $\delta_0$ if necessary, we may assume that for each
$x\in V_{n}\setminus V_{n+2}$, and $n\ge 0$,
\begin{equation}\label{eqn:der1}
\kappa_0^{-1} \diam (f(V_n))\ge \diam (V_n)|Df(x)|\ge \kappa_0\diam
(f(V_n)),
\end{equation}
and for each $c\in\Crit'(f)$ and $0<\delta<\delta'< r\delta_0$,
\begin{equation}
\kappa_0^{-1}\left(\frac{\delta'}{\delta}\right)^{1/\ell_c}\ge
\frac{\diam (\tB(c, \delta'))}{\diam (\tB(c,\delta))}\ge
\kappa_0\left(\frac{\delta'}{\delta}\right)^{1/\ell_c}
\end{equation}
where $\kappa_0$ is a constant depending only on $\ell_{\max}$.

\begin{lemma}\label{lem:crtret}
For each $n\ge 1$ and $c\in\Crit'(f)$, if $s$ is a positive integer
with $f^s(c)\in V_n\setminus V_{n+1}$ and with $f^j(c)\not\in
V_{n+1}$ for all $1\le j<s$, then
$$|Df^{s}(f(c))|\ge \kappa r,$$
where $\kappa>0$ is a constant depending only on $\ell_{\max}(f)$.
\end{lemma}
\begin{proof}

Let $c'\in \Crit'(f)$ be such that $f^{s}(c)\in V_n^{c'}$. For each
$j=0,1,\ldots, s$, let $U_j$ be the connected component of
$f^{-(s-j)}(\tB(c',\delta_{n-1}))$ which contains $f^j(c)$. Since
$f^j(c)\not\in V_{n+1}$ for all $1\le j<s$, and since $f$ is
backward contracting with constant $r>4$, we have $\diam (U_1)\le
\delta_{n-1}/r$, and $U_j\cap \Crit'(f)=\emptyset$ for all
$j=1,2,\ldots, s-1$. Thus $f^{s-1}: U_1\to U_{s}$ is conformal.
Since $f^{s}(c)\in V_n^{c'}$, there is a constant $\tau>0$ depending
on the order of $c'$ such that
$$\tB(c',\delta_{n-1})\supset B(f^s(c), \tau \diam
(\tB(c',\delta_{n-1})).$$ Applying the Schwarz lemma to the inverse
of $f^{s-1}: U_1\to U_{s}$, we obtain
\begin{align*}
|Df^{s}(f(c))|=|Df^{s-1}(f(c))||Df(f^s(c))|\ge \tau \frac{\diam
(\tB(c', \delta_{n-1}))}{\diam (U_1)}|Df(f^{s}(c))|,
\end{align*}
which implies by (\ref{eqn:der1}) that $|Df^s(f(c))|\ge
2\tau\kappa_0 r.$ Defining $\kappa=2\tau\kappa_0$ completes the
proof.
\end{proof}

\begin{lemma}\label{lem:shiftret} There exists $r_*>4$ depending only on $\ell_{\max}(f)$ such that
if $f$ satisfies $BC(r_*)$ then the following holds.  Let $x\in
V_{n+1}$ and let $s$ be a positive integer with $f^s(x)\in
V_n\setminus V_{n+1}$ and with $f^j(x)\not\in V_{n+1}$ for all
$j=1,2,\ldots, s-1$. Then $|Df^s(f(x))|\ge 1$.
\end{lemma}
\begin{proof} Assume that $f$ satisfies $BC(r_*)$ with $r_*$ sufficiently large so that
$$\frac{4\kappa_0^2}{3} \left(\frac{r_*}{2}\right)^{\ell_{c'}}\ge 1.$$
Let $c, c'\in\Crit'(f)$ be such that $x\in V_{n+1}^c$ and $f^s(x)\in
V_n^{c'}$. For each $j=0,1,\ldots, s$, let $U_j$  be the connected
component of $f^{-(s-j)}(\tB(c',r_* \delta_{n+1}))$ which contains
$f^j(x)$. Since $f^j(x)\not\in\tB(\Crit'(f),\delta_{n+1})$ for all
$1\le j<s$, $f^{s-1}: U_1\to U_s$ is conformal. Moreover,
$\diam(U_1)\le \delta_{n+1}$. Provided that $r_*$ is large enough,
$U_s$ contains a ball centered at $f^s(x)$ and of radius at least
$\diam (U_s)/3$. By the Schwarz lemma, we have
\begin{align*}
  |Df^s(f(x))|
& =|Df^{s-1}(f(x))||Df(f^s(x))|\\
& \ge \frac{\diam (U_s)}{3\diam (U_1)}|Df(f^s(x))|\\
& =\frac{\diam (\tB(c', r_*\delta_{n+1}))}{\diam
(V_n^{c'})}\frac{\diam(V_n^{c'})|Df(f^s(x))|}{3\diam (U_1)}\\
&\ge \kappa_0^2 \left(\frac{r_*}{2}\right)^{\ell_{c'}}\frac{\diam
(fV_n^{c'})}{3\diam (U_1)}\\
& \ge  \frac{4\kappa_0^2}{3}
\left(\frac{r_*}{2}\right)^{\ell_{c'}}\ge 1.
\end{align*}
\end{proof}

\iffalse
\begin{lemma}\label{lem:noncrtret}
Provided that $r$ is large enough, the following holds. Assume that
$x\in V_n\setminus V_{n+1}$, $f^s(x)\in V_n$ and $f^j(x)\not\in
V_{n+1}$ for all $1\le j<s$  then
$$|Df^s(f(x))|\ge 1.$$
\end{lemma}
\begin{proof} Let $c, c'\in \Crit'(f)$ be such that $x\in V_n^c$ and
$f^s(x)\in V_n^{c'}$. By assumption, $f^j(x)\not\in \tB(\Crit'(f),
\delta_{n+1})$ for all $j=0,1,\ldots, s-1$. Since $f$ satisfies the
backward contraction condition with constant $r$, then $f^s$ maps a
neighborhood $Q$ of $x$ diffeomorphically onto $\tB(c', r
\delta_{n+1})$. Moreover, since $x\in V_n^c$, we have $Q\subset
\tB(c, 2\delta_n)$. Let $P$ be subset of $Q$ which is mapped onto
$\tB(c', r\delta_{n+1}/2)$ by $f^s$. Then $f^s: P\to V_n^c$ has
uniformly bounded distortion. So there is a universal constant $C>0$
such that
\begin{align*}
|Df^s(f(x))|& =|Df^s(x)|\frac{|Df(f^s(x))|}{|Df(x)|}\\
& \ge C\frac{\diam (V_n^{c'})|Df(f^s(x))|} {\diam
(V_n^c)|Df(x)|}\frac{\diam
(V_n^c)}{\diam(\tB(c,2\delta_n))}\frac{\diam (\tB(c',
r\delta_{n+1}/2))}{\diam (V_n^{c'})}\\
& \ge C\kappa_0^2 \left(\frac{1}{2}\right)^{1/\ell_c}
\left(\frac{r}{4}\right)^{1/\ell_{c'}}.
\end{align*}
Thus $|Df^s(f(x))|\ge 1$ provided that $r$ is large enough.
\end{proof}
\fi
\begin{proof}[Proof of Theorem~\ref{thm:bc2ld}]
Let $r_0=\max (r_*, \kappa^{-1})$, where $\kappa$ is as in
Lemma~\ref{lem:crtret} and $r_*$ is as in Lemma~\ref{lem:shiftret}.

Assume that $f$ satisfies $BC(Kr_0)$ with $K\ge 1$. We shall prove
that for any $S\ge 1$ and $c\in\Crit'(f)$ with $f^S(c)\in V_0$,
$|Df^S(f(c))|\ge K$.

Given $S$ and $c$ as above, let us define inductively non-negative
integers $n_1>n_2>\cdots>n_m$ and $S_1<S_2<\cdots<S_m=S$ as follows.
First,
$$n_1=\max\{n\ge 0: f^j(c)\in V_{n}\mbox{ for some }1\le j\le S\},$$
and $$S_1=\max\{s\le S: f^s(c)\in V_{n_1}\}.$$ (Observe that the
backward contracting assumption on $f$  implies that $f$ has no
critical relation, so $n_1$, and hence $S_1$ is well-defined.) If
$S_1=S$ then we stop. Otherwise, let
$$n_2=\max\{n\ge 0: f^j(c)\in V_{n}\mbox{ for some }S_1<j\le S\},$$ and
$$S_2=\max\{s: S_1<s\le S, f^s(c)\in V_{n_2}\}.$$ Repeating the argument, we must stop within
finitely many steps.

By Lemma~\ref{lem:crtret}, $|Df^{S_1}(f(c))|\ge \kappa Kr_0\ge K.$
By Lemma~\ref{lem:shiftret}, for each $i=1,2,\ldots, m-1$ we have
$$|Df^{S_{i+1}-S_i}(f^{S_i+1}(c))|\ge 1.$$
Thus
$$|Df^s(f(c))|=|Df^{S_1}(f(c))|\prod_{i=1}^{m-1}
|Df^{S_{i+1}-S_i}(f^{S_i+1}(c))|\ge K.$$
\end{proof}

\appendix

\section{A priori bounds}\label{sec:bounds}

The goal of this section is to prove
Proposition~\ref{prop:aprioribounds}. Besides the results of
~\cite{KS}, we shall also use some arguments in~\cite[Section
6]{KSS}. Throughout this section, assume that $f$ is a polynomial,
at most finitely renormalizable and having only hyperbolic periodic
points.

\subsection{The puzzle construction} We shall now describe a puzzle
partition which will provide nice topological disks as required. The
construction given below is a modification of that in~\cite[Section
2]{KS} and makes use of equipotential curves for the Green function
and all bounded periodic Fatou components, external rays and
internal rays. The modification is necessary since we use the usual
definition of renormalization rather than the one used in \cite{KS},
and since we want to have an arbitrarily small puzzle piece for each
$c\in\Crit'(f)$.

Let $G$ denote the Green function of $f$. A {\em smooth external
ray} is a smooth gradient line of $G$ which starts from infinity and
tends to the Julia set of $f$. A smooth external ray has a
well-defined angle $t\in\R/\Z$ at which the ray goes to $\infty$. An
{\em external ray} is either a smooth external ray, or else a limit
of such rays. So for each $t\in\R/\Z$ there is exactly one smooth
external rays $\cR^t$, or two (non-smooth) external rays $\cR^{t,+}$
and $\cR^{t,-}$. Let $P_{\bad}$ be the set of periodic points which
are  contained in the forward orbit of either a critical point or
the landing point of a non-smooth external ray. Note that $P_{\bad}$
is a finite set. Let $P_{\sep}$ be the set of periodic points which
are the common landing point of at least $2$ external rays. (Here
``sep'' stands for ``separable''.)

 \noindent
{\bf Observation 1.} {\em Let $K'$ be a periodic component of the
filled Julia set of $f$ that contains a critical point but no
attracting periodic points. Then $K'\cap P_{\sep}$ is an infinite
set.}

In fact by \cite{LP}, we only need to consider the case that $f$ has
a connected Julia set. So $\partial K'=J(f)$ and $f$ has only
repelling periodic points. Let $N_n$ (resp. $N_{n}'$) be the number
of periodic points (resp. external rays) of $f$ which has period
$n$. Then $N_1=N_1'+1$ and for all $n>1$, $N_n=N_n'$. Let us
construct a sequence of integers $s_1 = 1 < s_2 < ...$, such that
for each $j \ge 1$, $f$ has a periodic point $p_j$ of period $s_j$
which is not the landing point of an external ray of the same
period. Since $N_1>N_1'$, there exists a fixed point of $f$ which is
not the landing point of an external ray fixed by $f$. Suppose now
that $p_j$ and $s_j$ have been defined. Let $\gamma_j$ be an
external ray landing at $p_j$ and let $s_{j+1}$ be the period of
$\gamma_j$. Clearly, $s_{j+1}>s_j\ge 1$. Since
$N_{s_{j+1}}=N'_{s_{j+1}}$, $f$ has a periodic point $p_{j+1}$ of
period $s_{j+1}$ which is not the landing point of an external ray
of period $s_{j+1}$. This proves the existence of
$\{s_j\}_{j=1}^\infty$ and $\{p_j\}_{j=1}^\infty$. Since $p_j\in
P_{\sep}$ for all $j\ge 1$, the statement follows.

Let $\cK$ be the collection of all periodic components of the filled
Julia set of $f$ that contain a critical point. For each $K'\in\cK$,
choose $a_{K'}\in (K'\cap P_{\sep})\setminus P_{\bad}$ and let
$\Xi_{K'}$ be the union of the orbit of $a_{K'}$ and all external
rays landing on this orbit. Moreover, let
$$\Theta_1=\bigcup_{K'\in\cK} \Xi_{K'}.$$

Now we shall define for each bounded Fatou component $B$, an {\em
equipotential curve} $\Gamma_B$ and {\em internal rays}
$\gamma_B^\theta$, $\theta\in\R/\Z$. We start by choosing
 bounded periodic Fatou components $B_1,B_2,\ldots, B_m$, such that
the orbits of $B_i$'s are pairwise disjoint and such that the grand
orbit of $\bigcup_{i=1}^m B_i$ covers the interior of the filled
Julia set. By assumption, $f$ has an attracting periodic point
$p_i\in B_i$ with period $s_i\ge 1$. We choose Jordan curves
$\Gamma_{B_i}\subset B_i$ such that $B_i\setminus \Omega_i$ is
disjoint from the orbit of all critical points and such that
$\Gamma_{B_i}':=f^{-s_i}(\Gamma_{B_i})\cap B_i$ lies in the interior
of $B_i\setminus \Omega_i$, where $\Omega_i$ is the topological disk
bounded by $\Gamma_{B_i}$.  Then $\Gamma_{B_i}'$ is also a Jordan
curve which bounds a topological disk $\Omega_i'$. Denote by $d_i$
the degree of the map $f^{s_i}: B_i\to B_i$. Choose a diffeomorphism
$$h_{B_i}: B_i\setminus \Omega_i\to \{1<|z|\le 2^{d_i}\}$$ such that
\begin{itemize}
\item $h_{B_i}(\Gamma_{B_i})=\{|z|=2^{d_i}\}$ and $h_{B_i}(\Gamma_{B_i}')=\{|z|=2\}$;
\item $h_{B_i}(f^{s_i}(z))=h_{B_i}(z)^d$ holds for all $z\in B_i\setminus
\Omega_i'$.
\end{itemize}
For $\theta\in\R/\Z$, define $\gamma_{B_i}^\theta:=h_{B_i}^{-1}
(\{re^{2\pi i\theta}: 1<r\le 2^{d_i}\})$. Given a bounded Fatou
component $B$ there exists a unique $i\in\{1,2,\ldots, m\}$ and a
minimal non-negative integer $n_B\ge 0$ such that $f^{n_B}(B)=B_i$.
Let $\Gamma_{B}=f^{-n_B}(\Gamma_{B_i})\cap B$ which is also a Jordan
curve, and define internal rays
$$\gamma_{B}^{\theta}=f^{-(n_B-i)}(\gamma_{B_i}^{\theta})\cap B.$$

\noindent {\bf Observation 2.} {\em There exists $\theta_0\in\R/\Z$
which is periodic under the map $\theta\to d_i\theta$, such that
$\gamma_{B_i}^{\theta_0}$ converges to a periodic point in $\partial
B_i \setminus P_{\bad}$.}

In fact, as in~\cite[Lemma 2.1]{KS}, we can prove that for each
$\theta\in\R/\Z$ which is periodic under the map $\theta\to
d_i\theta$, $\gamma_{B_i}^{\theta}$ converges to a periodic point in
$\partial B_i$. Moreover, for two distinct $\theta$,
$\gamma_{B_i}^\theta$ converges to different points. Since
$P_{\bad}$ is finite, the statement follows.

For each $i=1,2,\ldots, m$, let us fix $\theta_i\in\R/\Z$ which is
periodic under $t\mapsto d_i t$ such that $\gamma_{B_i}^{\theta_i}$
converges to a periodic point $a_i\in \partial B_i\setminus
P_{\bad}$. Let
$$\Xi_i=\bigcup_{k=0}^{s_i-1} \left(\Gamma_{f^k(B_i)} \cup
\bigcup_{j=0}^\infty \gamma_{f^k(B_i)} ^{d_i^j \theta_i}\right),$$
let $\Xi_i'$ be the union of the orbit of $a_i$ and the external
rays landing on this orbit, and let
$$\Theta_2=\bigcup_{i=1}^m (\Xi_i\cup \Xi_i').$$

Now we are ready to construct the puzzle. Let $\eps>0$ be such that
the equipotential set $\Theta_0:=\{G(z)=\eps\}$ is disjoint from the
grand orbit of all critical points. Let $\Theta=\Theta_0\cup
\Theta_1\cup\Theta_2$. This is a finite union of smooth curves and
periodic points, and each of these periodic points is the landing
point of two or more smooth (external or internal) rays. Let
$\mathcal{P}_0$ be the collection of all components of $\C\setminus
\Theta$ which intersect the Julia set of $f$ and for $n\ge 1$, let
$\mathcal{P}_n$ be the collection of components of $f^{-n}(P)$,
where $P$ runs over all elements of $\mathcal{P}_0$. An element of
$\mathcal{P}_n$ is called a {\em puzzle piece of depth $n$}.

%Let $N_0$ be the maximal period of these periodic points. Let $K'$
%be a periodic component of the filled Julia set that contains a
%critical point.
%\begin{itemize}
%\item If $K'$ does not contain an attracting period point, then we choose
%a periodic point $a_{K'}\in K'$ with period greater than $N_0$ such
%that there are at least two external rays landing at $a_{K'}$. Let
%$\Gamma_{K'}$ be the union of the orbit of $a_{K'}$ and all external
%rays landing on this orbit.
%\item For each $i=1,2,\ldots, m$, choose $\theta_i\in\R/\Z$ such
%that it is periodic under $t\mapsto d_i t$ of period greater than
%$N_0$. Let
%$$\Gamma_{i}=\bigcup_{k=0}^{s_i-1} \bigcup_{j=0}^\infty \gamma_{f^k(B_i)}^{d^j\theta_i}$$
%\end{itemize}
Since $\Theta$ is disjoint from the orbit of critical points, for
each $c\in\Crit'(f)$ and each $n\ge 0$ there is a puzzle piece
$P_n(c)$ of depth $n$ that contains $c$. If for each $c\in\Crit'(f)$
and each integer $s\ge 1$, there exists $n\ge 0$ such that
$f^s(c)\not\in P_n(c)$, then the arguments in \cite{KS} show that
$\diam (P_n(c))\to 0$ as $n\to\infty$ for each $c\in\Crit'(f)$.

Otherwise, there exists $c\in\Crit(f)$ and a minimal integer $s\ge
1$ such that $f^s(c)\in P_n(c)$ for all $n\ge 0$. The degree of
$f^s: P_{n+s}(c)\to P_n(c)$ is non-increasing, hence eventually
constant. So there exists $n_0$ such that all the critical points of
$f^s: P_{n_0+s}(c)\to P_{n_0}(c)$ do not escape $P_{n_0+s}(c)$ under
forward iteration of this map. Using the ``thickening'' technique
\cite{Milnor}, we can find topological disks
$\widehat{P}_{n_0+s}(c)\supset P_{n_0+s}(c)$ and
$\widehat{P}_{n_0}(c)\supset P_{n_0}(c)$ such that $f^s:
\widehat{P}_{n_0+s}(c)\to\widehat{P}_{n_0}(c)$ is a renormalization
of $f$. Then arguing as in Observation 1, there exists a period
point $q\in P_{n_0+s}(c)\cap (P_{\sep}\setminus P_{\bad})$.  We add
the orbit of $q$ and the external rays landing on the orbit of $q$
in the set $\Theta$ and construct a new puzzle. Then either for each
$c\in\Crit'(f)$ we obtain an arbitrarily small new puzzle pieces
containing $c$, or we obtain a new renormalization of $f$. The new
renormalization either has a larger period, or has a smaller degree.
Repeating the argument, if $f$ is not infinitely renormalizable,
then we must stop within finitely steps.

In conclusion, if $f$ is at most finitely renormalizable and has
only hyperbolic periodic points, then we can construct a puzzle such
that for each $c\in\Crit'(f)$ there exists an arbitrarily small
puzzle pieces containing $c$. In the following we fix such a puzzle.

\subsection{The bounds}

Given a recurrent critical point $c$ of $f$ and a nice topological
disk $V\ni c$, we say that a topological disk $U$ is a {\em child}
of $V$ if there exist $c'\in [c]$ and $s\ge 1$ such that $U\ni c'$
and such that $f^s: U\to V$ is a proper map with a unique critical
point, where
$$[c]=\{\zeta\in\Crit'(f):\omega(c)=\omega(\zeta)\ni c, \zeta\}.$$
 We say that $c$ is {\em persistently
recurrent} if for each $c'\in [c]$, each nice topological disk $V\ni
c'$ has only finitely many children. Otherwise, we say that the
recurrent critical point $c$ is {\em reluctantly recurrent}.

Let $\cV$ be a nice topological disk which contains a critical point $c\in\Crit'(f)$. We say that
$\cV$ is {\em essentially $\rho$-nice} if for each return domain $U$ of $\cV$ which
intersects $\text{orb}(c)$, we have $\mod(\cV; U)\ge \rho.$

\begin{lemma}\label{lem:persistent}
For each $c\in \Crit'(f)$ which is persistently recurrent, there
exists an arbitrarily small puzzle piece that contains $c$ and  is
essentially $\rho_1$-nice, where $\rho_1>0$ is a constant
depending only on the degree of $f$.
\end{lemma}
\begin{proof} Given a small puzzle piece $W$ that contains $c$, we can
construct a complex box mapping as in Lemma 2.2 of~\cite{KS}, but we
take $V$ to be the union of the components of the domain of the
first entry map to $W$ which intersect $[c]$. Then the box mapping
has only persistently recurrent critical points. Then by the
proposition 10.1 of~\cite{KS}, the $\textbf{I}_n$ constructed in \S
8 of~\cite{KS} is essentially $\rho_1$-nice for some universal
$\rho_1>0$ when $n$ is large enough.
\end{proof}

Let us say that an open set $\cV$ is a {\em puzzle neighborhood} of
a set $A\subset \Crit'(f)$ if $\cV$ is nice and each component of
$\cV$ is a puzzle piece containing exactly one point of $A.$ As
before, If $\cV'\supset \cV$ are two puzzle neighborhoods of $A$
such that $$f^n(\partial \cV_a)\cap \cV_a'=\emptyset,$$ holds for
all $n\ge 1$ and $a\in A$, where $\cV_a$ (resp. $\cV_a'$) denotes
the component of $\cV$ (resp. $\cV'$) which contains $a$, then we
say that $(\cV', \cV)$ is an  admissible pair. We say that $a\in A$
is {\em special} for the pair $(\cV', \cV)$ if
$$f^n(\cV_a)\cap \cV'=\emptyset\mbox{ for all } n\ge 1.$$
Recall that $(\cV', \cV)$ is called $\rho$-bounded if for each $a\in
A$,
$$\mod(\cV_a'; \cV_a)\ge \rho.$$

\begin{lemma} \label{lem:adm}
For any $\rho>0$ there exists $\rho'>0$ such that the
following holds. Let $c\in\Crit'(f)$ and let $(\cV', \cV)$ be an
$\rho$-bounded admissible pair  of puzzle neighborhoods of
$\Back(c)$. Assume that $\cV'$ is sufficiently small. Then for each
$x\in D(\cV)$, we have
$$\mod(\sL_x(\cV');\sL_x(\cV))\ge \rho'.$$
Moreover, if $c$ is special for $(\cV', \cV)$ then $\cV_c$ is
$\rho'$-nice.
\end{lemma}
\begin{proof} Take $x\in D(\cV)$.
Let $s$ be the entry time of $x$ into $\cV$, let $\cU=\sL_x(\cV)$.
Let $a\in \Back(c)$ be such that $f^s(x)\in \cV_a$ and let $\cU'$ be
the component of $f^{-s}(\cV_a')$ which contains $x$. By
Lemma~\ref{lem:admdeg}, the degree of $f^s: \cU'\to \cV'_a$ is
bounded from above by a constant $N_0$. Thus $\mod(\cU'; \cU)\ge
N_0^{-1} \rho$. Since $\cU'\subset \sL_x(\cV')$, $\mod(\sL_x(\cV'),
\sL_x(\cV))\ge \mod(\cU'; \cU)\ge \rho/N_0$.

Now assume that $c$ is special. Then for $x\in D(\cV_c)\cap \cV_c$,
we have $\sL_x(\cV')\subset \cV_c$. Clearly, $\sL_x(\cV_c)\subset
\sL_x(\cV)$. Thus
$$\mod(\cV_c; \sL_x(\cV_c))\ge \mod(\sL_x(\cV'); \sL_x(\cV))\ge \rho/N_0.$$
This proves the last statement.
\end{proof}

\begin{lemma}\label{lem:admtran}
For any $\rho>0$ there exists $\rho'>0$ such that the following
holds. Let $c\in\Crit'(f)$ and let $(V',V)$ be a $\rho$-bounded
admissible pair of puzzle neighborhoods of $c$. Suppose that for
each $x\in \bigcup_{c'\in\Back(c)} \overline{\textrm{orb}(c')}$,
either $\sL_x(V)=\emptyset$ or
$$\mod(V',\sL_x(V'))\ge \rho.$$
If $\diam (V')$ is sufficiently small, then
\begin{enumerate}
\item [(i)] there exists a $\rho'$-bounded admissible pair $(\cV', \cV)$ of puzzle
neighborhoods of $\Back(c)$ such that $\cV'_c=V'$ and $\cV_c=V$ and
such that $c$ is special for this admissible pair;
\item [(ii)] for any $x\in D(V)$, we have $\mod(\sL_x(V'); \sL_x(V))\ge \rho'$;
\item [(ii)] $V$ is $\rho'$-nice.
\end{enumerate}
\end{lemma}
\begin{proof}
(i) Let $c_0=c, \cV_0'=V', \cV_0=V$ and let $c_1, c_2, \ldots,
c_{b-1}$ be the critical points in $\Back(c)\setminus\{c\}$. For
each $i=1,2,\ldots, b-1$, let $\cV_i'$ be the entry domain of $V'$
that contains $c_i$ and let $t_i$ be the entry time. If
$f^{t_i}(c_i)\in V$, let $\cV_i$ be the component of $f^{-t_i}(V)$
that contains $c_i$. Otherwise, let $\cW_i$ be the entry domain of
$V'$ that contains $f^{t_i}(c_i)$ and let $\cV_i$ be the component
of $f^{-t_i}(\cW_i)$ that contains $c_i$. Put
$\cV'=\bigcup_{i=0}^{b-1} \cV_i'$ and $\cV=\bigcup_{i=0}^{b-1}
\cV_i$. Clearly, $(\cV', \cV)$ is an admissible pair of puzzle
neighborhoods of $\Back(c)$, and $c$ is a special critical point.
Since the maps $f^{t_i}: \cV_i'\to V'$ has uniformly bounded degree,
the admissible pair $(\cV', \cV)$ is $\rho'$-bounded, where
$\rho'>0$ is a constant.

(ii) By Lemma~\ref{lem:adm} and redefining $\rho'>0$,  we obtain
that for each $x\in D(\cV)$,
$$\mod(\sL_x(\cV'),\sL_x(\cV))\ge \rho'.$$
Since $\cV_i'$ is an entry domain of $\cV_0'$ for each $i=1,2,\ldots, b-1$, we have
$$ \sL_x(\cV')=\sL_x(\cV_0')=\sL_x(V').$$
For $x\in D(V)$,  $\sL_x(\cV)\supset \sL_x(V)$. It follows that
$$\mod(\sL_x(V'); \sL_x(V))\ge\mod(\sL_x(\cV'),\sL_x(\cV))\ge \rho'.$$

(iii) Since $c$ is special for the pair $(\cV', \cV)$, applying the
last statement of Lemma~\ref{lem:adm} proves the result.
\end{proof}
%
%\begin{proof}
%Let $c_0=c, V_0=V, U_0=U$ and let $c_1, c_2, \ldots, c_{b-1}$ be the
%critical points in $\Back(c)\setminus\{c\}$. For each $i=1,2,\ldots,
%b-1$, let $V_i$ be the entry domain of $V$ that contains $c_i$ and
%let $t_i$ be the entry time. If $f^{t_i}(c_i)\in U$, let $U_i$ be
%the component of $f^{-t_i}(U)$ that contains $c_i$. Otherwise, let
%$W_i$ be the entry domain of $V$ that contains $f^{t_i}(c_i)$ and
%let $U_i=\comp_{c_i}f^{-t_i}(W_i)$. Then $(\bigcup_{i=0}^{b-1} V_i,
%\bigcup_{i=0}^{b-1} U_i)$ is an admissible puzzle neighborhood of
%$\Back(c)$, and $c_0$ is a special critical point. Note that
%$\mod(V_i; U_i)\ge \lambda/ d_i$, where $d_i$ is the order of $c_i$.
%By Lemma~\ref{lem:adm}, the lemma follows.
%\end{proof}

\begin{lemma} \label{lem:reluctant}
If $c\in \Crit'(f)$ is reluctantly recurrent. Then for any
$\rho>0$, there exists an arbitrarily small puzzle neighborhood
$V_k$ of $c$ which is $\rho$-nice. In particular, $V_k$ is
essentially $\rho$-nice.
\end{lemma}
\begin{proof} By~\cite[Lemma 6.5]{KSS}, there exists a puzzle piece $V\ni c$, a positive integer $N$ and
a sequence of integers $s_k\to\infty$ with the following properties:
\begin{itemize}
\item $V$ is $\lambda$-nice for some $\lambda>0$;
\item $f^{s_k}(c)\in V$, and letting $V_k=\comp_c(f^{-s_k}V)$, we have
\item $f^{s_k}: V_k\to V$ has degree at most $N$.
\end{itemize}
By replacing $V$ with some pull back of $V$, we may also assume
$$ f^j(V_k)\not\ni c, \mbox{ for all } 1\le j<s_k,$$
which implies that
\begin{equation}\label{eqn:landingtime}
f^j(V_k)\cap V_k=\emptyset \mbox{ for all } 1\le j<s_k.
\end{equation}
Let $$\lambda_k=\inf_{y\in V\cap D(V_k)} \mod(V;\widehat{\sL}_y(V_k)),$$
where $\widehat{\sL}_y(V_k)=V_k$ if $y\in V_k$ and $\widehat{\sL}_y(V_k)=\sL_{y}(V_k)$ otherwise.
Since $\widehat{\sL}_y(V_k)\subset \sL_y(V)$ for each $y\in V\cap D(V_k)$, we have
$\lambda_k\ge \lambda$.

\noindent
{\bf Claim.} Each $V_k$ is $\lambda_k/N$-nice.

To prove this claim, take $x\in V_k\cap D(V_k)$. By (\ref{eqn:landingtime}), the return time of $x$ into $V_k$ is at least $s_k$ and
hence $f^{s_k}(\sL_x(V_k))=\widehat{\sL}_y(V_k)$, where $y=f^{s_k}(x)$.
Therefore
$$\mod(V_k; {\sL_x(V_k)})\geq N^{-1}\mod(V; {\widehat{\sL}_y(V_k)})\ge \lambda_k/N.$$

Letting
$U_k=\sL_c(V_k)$, and applying Lemma~\ref{lem:admtran}, we obtain
that for each $x\in D(U_k)$, $\mod(\sL_x(V_k); \sL_x(U_k))\ge
\rho_1$, where $\rho_1>0$ is a constant. It follows that $\lambda_k\to\infty$.
The lemma follows.
\end{proof}

\begin{proof}[Proof of Proposition~\ref{prop:aprioribounds}]
By Lemmas~\ref{lem:persistent} and~\ref{lem:reluctant}, there exists
a universal constant $\rho_1>0$ such that any $c\in\Crit'(f)$,
there exists an arbitrarily small puzzle piece that contains $c$ and
is essentially $\rho_1$-nice (note that by the definition of
essentially $\rho_1$-nice, the statement is trivial for
non-recurrent $c\in\Crit'(f)$). By Lemma~\ref{lem:admtran} (i), it
follows that there exists an arbitrarily small $\rho_2$-bounded
admissible pair $(\cV', \cV)$ of $[c]$ for which $c$ is special ,
where $\rho_2>0$ is a constant.

Let us define a strictly increasing finite sequence
$\{\Omega_k\}_{k=0}^b$ of subsets of $\Crit'(f)$ as follows. Let
$\Omega_0=\emptyset$. If $\Omega_k$ is defined and
$\Omega_k\not=\Crit'(f)$, then we proceed to define $\Omega_{k+1}$
by taking $c\in \Crit'(f)\setminus \Omega_k$ with $\Back(c)\subset
[c]\cup \Omega_k$ and letting $\Omega_{k+1}=\Omega_k\cup [c]$.
Clearly the procedure stops within $\#\Crit'(f)$ steps, so $b\le \#
\Crit'(f)$ and $\Omega_{b}=\Crit'(f)$. Note that for each $k$, any
$c\in\Omega_k$, $\Back(c)\subset \Omega_k$.

We claim that for each $k=1,2,\ldots, b$ and any $c\in
\Omega_k\setminus \Omega_{k-1}$, there exists an arbitrarily small
$\rho_2$-bounded admissible pair $(V', V)$ of puzzle neighborhoods of
$\Omega_k$ for which $c$ is special. Let us prove this claim by
induction. The case $k=1$ has been proved above, so assume that the
claim holds for $k=k_0-1$, $2\le k_0\le b$. Take $c\in
\Omega_{k+1}\setminus \Omega_k$. By induction hypothesis, for any $n_0\ge 1$ there is a
$\rho_2$-bounded admissible pair of puzzle neighborhoods $(\cW',\cW)$ of $\Omega_k$ such that
the depth of $\cW'_{c'}$ is greater than $n_0$ for each $c'\in\Omega_k$.
Let us choose a $\rho_2$-bounded admissible pair $(\cU', \cU)$
of puzzle neighborhoods of $[c]$ for which $c$
is special, and such that the minimal depth of components of $\cU'$ is greater than the maximal depth of components of $\cW$.
Then
$(\cW'\cup \cU', \cW\cup \cU)$ is a $\rho_2$-bounded admissible
puzzle neighborhood of $\Omega_{k+1}$. Note that
$f^n(\partial \cU)\cap \cW'=\emptyset$ for all $n\ge 0$, for otherwise, we would obtain that $c\in\Back(c')$
for some $c'\in\Omega_k$. Thus $c$ is special for
this admissible pair $(\cW'\cup \cU', \cW\cup \cU)$. This completes the induction step and thus the
proof of the claim.

In particular, for any $c\in\Crit'(f)$, there exists an arbitrarily
small  $\rho_2$-bounded admissible puzzle neighborhoods $(\cV',
\cV)$ of $\Back(c)$ for which $c$ is special. By Lemma~\ref{lem:adm}, $\cV_c$ is a
$\rho_0$-nice puzzle piece for some universal constant $\rho_0>0$.
\end{proof}

\section{Interval maps}\label{sec:interval}
We shall prove corresponding results for a class of interval maps.

Recall that a map~$f : X \to X$ from a compact interval~$X$ of~$\R$
into itself is \textit{of class~$C^3$ with non-flat critical points}
if~$f$ is of class~$C^1$ on~$X$; of class~$C^3$ outside $\Crit(f) :=
\{x\in X \mid Df(x)=0\}$; and for each $c \in \Crit(f)$, there
exists a number $\ell_c>1$ (called {\em the order of~$f$ at~$c$})
and diffeomorphisms $\phi, \psi$ of~$\R$ of class~$C^3$ with
$\phi(c)=\psi(f(c))=0$ such that,
$$|\psi\circ f(x)|=|\phi(x)|^{\ell_c}$$
holds in a neighborhood of~$c$ in~$X$.

For such a map $f$, we define~$J(f)$ to be the complement of the
interior of the attracting basins of periodic attractors and
$\Crit'(f)=\Crit(f)\cap J(f)$. As in \S\ref{sec:intro}, we can
define the properties $BC(r)$ and $LD(K)$ for $C^3$ interval maps
with non-flat critical points. Let
$$\ell_{\max}=\max_{c\in\Crit'(f)} \ell_c.$$

For interval maps, we have

\medskip

\noindent {\bf Theorem A'.} {\em For each $\ell>1$, there exists
$K_0=K_0(\ell)>0$ such that if $f$ is a $C^3$ interval maps with
$\#\Crit'(f)=N$ and with $\ell_{\max}\le \ell$, and if $f$ satisfies
$LD(K_0r)$ for some $r>1$, then $f$ satisfies $BC(r)$.}

\medskip

This is proved in~\cite[Theorem 1]{BRSS}, although the statement
here is slightly more general. We observe that ~\cite[Proposition
1]{BRSS}) remains true under the more general assumption here, if we
require that $f^s(T)$ is contained a small neighborhood of
$\Crit'(f)$ (which is given by \cite[Theorem C]{SV}). The dependence
of constants follows from the proof.

We also have the following

\medskip

\noindent {\bf Theorem B'.} {\em Let $f$ be  a $C^3$ interval maps
with non-flat critical points. There exists a constant $r_0>1$
depending only on $\ell_{\max}$ such that if $f$ satisfies
$BC(Kr_0)$ for some $K>1$, then $f$ satisfies $LD(K)$.}

\medskip
The proof of this theorem follows the same outline as that of
Theorem~\ref{thm:bc2ld}, replacing the Schwarz lemma by the
following real version.

\medskip
\noindent {\bf Real Schwarz Lemma.} {\em Let $f$ be  a $C^3$
interval maps with non-flat critical points. There exists
$\eta=\eta(f)>0$ and a universal constant $\theta\in (0,1)$ such
that if $f^n: U\to V$ is a diffeomorphism between intervals,
$V\subset \tB(c, \eta)$ for some $c\in\Crit'(f)$ and $x\in U$ is
such that $f^n(x)$ is the middle point of $V$, then
$$|Df^n(x)|\ge \theta\frac{|V|}{|U|}.$$
}
\begin{proof} Let $\hV$ be the open interval with $f^n(x)$ as middle
point and with $|\hV|=|V|/2$, and let $\hU=f^{-n}(\hV)\cap U$. By
\cite[Theorem C(2)]{SV}, there exists a constant $K>1$ such that
$$|Df^n(x)|\ge K^{-1} \frac{|\hV|}{|\hU|}\ge
(2K)^{-1}\frac{|V|}{|U|},$$ provided that $\eta$ is sufficiently
small. Thus the lemma holds with $\theta=1/(2K)$.
\end{proof}

\noindent
Department of Mathematics \\
University of Science and Technology of China\\
Hefei 230026, CHINA

\medskip
\noindent hbli@mail.ustc.edu.cn

\bigskip
\noindent
Department of Mathematics \\
National University of Singapore\\
Singapore 117543, SINGAPORE

\medskip
\noindent matsw@nus.edu.sg

\end{document}